\documentclass[12pt,reqno]{amsart}

\usepackage{verbatim}

\newtheorem{thm}{Theorem}

\newtheorem{lem}[thm]{Lemma}
\theoremstyle{definition}

\providecommand{\RR}{\mathbb{R}}

\providecommand{\NN}{\mathbb{N}}

\providecommand{\eps}{\epsilon}

\def\eps{\varepsilon}
\begin{document}

\title[Positivstellens\" atze for matrix polynomials]{Strict Positivstellens\" atze for matrix polynomials with scalar constraints}

\author{J. Cimpri\v c}



\date{submitted August 1st 2010, revised October 31st 2010}

\address{Jaka Cimpri\v c, University of Ljubljana, Faculty of Math. and Phys.,
Dept. of Math., Jadranska 19, SI-1000 Ljubljana, Slovenija. 
E-mail: cimpric@fmf.uni-lj.si. www page: http://www.fmf.uni-lj.si/ $\!\!\sim$cimpric.}

\begin{abstract}
We extend Krivine's strict positivstellensatz for usual (real multivariate) polynomials 
to symmetric matrix polynomials with scalar constraints.
The proof is an elementary computation with Schur complements. Analogous extensions 
of Schm\" udgen's and Putinar's strict positivstellensatz were recently proved by
Hol and Scherer using methods from optimization theory. 
\end{abstract}

\maketitle

\thispagestyle{empty}

\section{Introduction}

Let $S=\{g_1,\ldots,g_m\}$ be a finite subset of the algebra $\RR[\underline{X}]=\RR[X_1,\ldots,X_d]$. Write
$$K_S = \{x \in \RR^d \mid g_1(x) \ge 0, \ldots, g_m(x) \ge 0\}$$
and 
$$M_S = \{c_0+\sum_{i=1}^m c_i g_i \mid c_0,\ldots,c_m \in \sum \RR[\underline{X}]^2\}.$$
Write also $\widehat{S}=\{g_1^{\alpha_1} \cdots g_m^{\alpha_m} \mid \alpha_1,\ldots,\alpha_m \in \{0,1\}\}$ 
and $T_S=M_{\widehat{S}}$. 

The following theorem summarizes the strict positivstellens\" atze
of Krivine \cite{kr}, \cite{st}, Schm\" udgen \cite{sch}, \cite{bw}
and Putinar \cite{pu}, \cite{jac} (respectively, (1) $\Leftrightarrow$ (2), (1) $\Leftrightarrow$ (2') and (1) $\Leftrightarrow$ (2'')). 
A nice overview is \cite{mm}.

\begin{thm}
\label{ctm}
Notation as above. For every $f \in \RR[\underline{X}]$ the following are equivalent:
\begin{enumerate}
\item $f(x) > 0$ for every $x \in K_S$,
\item there exist $t,u \in T_S$ such that $(1+t)f=1+u$.
\end{enumerate}
If $K_S$ is compact then (1) and (2) are equivalent to
\begin{enumerate}
\item[(2')] there exists an $\eps >0$ such that $f-\eps \in T_S$.
\end{enumerate}
If one of the sets $K_{\{g_i\}} = \{x \in \RR^d \mid g_i(x) \ge 0\}$, $i=1,\ldots,m$, is compact,
then (1), (2) and (2') are equivalent to
\begin{enumerate}
\item[(2'')] there exists an $\eps >0$ such that $f-\eps \in M_S$.
\end{enumerate}
\end{thm}

Usually one has $tf=1+u$ in (2). Our version is then a consequence of
$(1+t+u)f=tf+(1+u)f=1+u+tf^2$, see \cite[p. 26]{mm}.

Let $M_n(\RR[\underline{X}])$ be the algebra of all $n \times n$ matrices with entries from the algebra $\RR[\underline{X}]$.
Write $\sum M_n(\RR[\underline{X}])^2$ for the set of all finite sums of elements of the form $A^TA$ where $A \in M_n(\RR[\underline{X}])$.
For $S$  as above write
$$M_S^n = \{C_0+\sum_{i=1}^m C_i g_i \mid C_0,\ldots,C_m \in \sum M_n(\RR[\underline{X}])^2\}$$
and $T_S^n=M_{\widehat{S}}^n$. Clearly $M_S^n$ is a quadratic module (i.e. $M_S^n$ contains the identity matrix $I_n$, $M_S^n+M_S^n \subseteq M_S^n$ and
$A^T M_S^n A \subseteq M_S^n$ for every $A \in M_n(\RR[\underline{X}])$). The quadratic module $T_S^n$ 
also satisfies $T_S \cdot T_S^n \subseteq T_S^n$.

The aim of this note is to prove the equivalence (1) $\Leftrightarrow$ (2) in the following theorem. The equivalences 
(1) $\Leftrightarrow$ (2') and (1) $\Leftrightarrow$ (2'') just rephrase the Hol-Scherer theorem \cite[Theorem 2]{hs}
and are stated here for the sake of completeness.

\begin{thm} 
\label{mst}
Notation as above.
For every element $F \in M_n(\RR[\underline{X}])$ such that $F^T=F$, the following are equivalent:
\begin{enumerate}
\item $F(x)$ is strictly positive definite for every $x \in K_S$,
\item there exist $t\in T_S$ and $V \in T_S^n$ such that $(1+t)F=I_n+V$.
\end{enumerate}
If $K_S$ is compact then (1) and (2) are equivalent to
\begin{enumerate}
\item[(2')] there exists an $\eps >0$ such that $F-\eps I_n \in T_S^n$.
\end{enumerate}
If one of the sets $K_{\{g_i\}} = \{x \in \RR^d \mid g_i(x) \ge 0\}$, $i=1,\ldots,m$, is compact,
then (1), (2) and (2') are equivalent to
\begin{enumerate}
\item[(2'')] there exists an $\eps >0$ such that $F-\eps I_n \in M_S^n$.
\end{enumerate}
\end{thm}

\section{The proof}

We will need the following technical lemma:

\begin{lem}
\label{lmst}
For every $B \in M_n(\RR[\underline{X}])$ there exists $c \in \sum \RR[\underline{X}]^2$ such that
$c I_n - B^T B \in \sum M_n(\RR[\underline{X}])^2$.
\end{lem}

\begin{proof} 
We can take $c$ to be of the form $k p^l$ where $p=1+\sum_{i=1}^d X_i^2$ and $k,l$ are positive integers.
Namely, let $L$ be the set of all $B \in M_n(\RR[\underline{X}])$ such that 
$k p^l I_n - B^T B \in \sum M_n(\RR[\underline{X}])^2$ for some positive integers $k$ and $l$.
Clearly, $L$ contains $X_1,\ldots,X_d$ and all constant matrices. To prove that $L=M_n(\RR[\underline{X}])$
it suffices to show that $L$ is closed for addition and multiplication. Suppose that $B_1,B_2 \in L$.
There exist positive integers $k_1,k_2,l_1,l_2$ such that 
$$k_1 p^{l_1}I_n-B_1^T B_1\in \sum M_n(\RR[\underline{X}])^2, \quad k_2 p^{l_2}I_n-B_2^T B_2\in \sum M_n(\RR[\underline{X}])^2. $$
The paralelogram identity
%$(B_1+B_2)^T(B_1+B_2)+(B_1-B_2)^T(B_1-B_2)=2(B_1^T B_1+B_2^T B_2)$ 
implies that
$$2(k_1+k_2)p^{\max(l_1,l_2)}I_n-(B_1+B_2)^T(B_1+B_2) \in \sum M_n(\RR[\underline{X}])^2 $$
and by inserting terms $\pm k_2 p^{l_2} B_1^T B_1$ we see that
$$k_1 k_2 p^{l_1+l_2}I_n - (B_1 B_2)^T(B_1 B_2)^T\in \sum M_n(\RR[\underline{X}])^2.$$
\end{proof}

We can now return to the proof of Theorem \ref{mst}.

\begin{proof}

Clearly, (2'') $\Rightarrow$ (2') $\Rightarrow$ (2) $\Rightarrow$ (1) (with no assumptions on $K_S$).

The implications (2') $\Rightarrow$ (1) when $K_S$ is compact and (2'') $\Rightarrow$ (1)
when one of $K_{\{g_i\}}$ is compact follow from Theorem \ref{ctm} and Hol-Scherer theorem
\cite[Theorem 2]{hs}.

We will now prove that (2) implies (1) (with no assumptions on $K_S$) by induction on $n$. 
The case $n=1$ is covered by Theorem \ref{ctm}. Suppose that
(1) implies (2) for all symmetric matrix polynomials of size $n-1$
and pick a symmetric polynomial $F(x)$ of size $n$ which satisfies (1).
We write 
$$F= 
\left[ \begin{array}{cc} f_{11} & g \\ g^T & H  \end{array} \right]
$$
and observe that (in $M_n(\RR(X))$)
\begin{equation}
%\left[ \begin{array}{cc} 1 & 0 \\ - \frac{1}{f_{11}}g^T & I_{n-1} \end{array} \right]
\left[ \begin{array}{cc} 1 & -\frac{1}{f_{11}}g \\ 0 & I_{n-1} \end{array} \right]^T 
\left[ \begin{array}{cc} f_{11} & g \\ g^T & H  \end{array} \right]
\left[ \begin{array}{cc} 1 & -\frac{1}{f_{11}}g \\ 0 & I_{n-1} \end{array} \right] =
\left[ \begin{array}{cc} f_{11} & 0 \\ 0 & \tilde{H} \end{array} \right]
\end{equation}
where $\tilde{H} = H-\frac{1}{f_{11}}g^T g$ is the Schur complement of $f_{11}$. 
Since $F$ is stricly positive definite on $K_S$, it follows that $f_{11}$
is stricly positive on $K_S$, hence 
$\left[ \begin{array}{cc} 1 & -\frac{1}{f_{11}}g \\ 0 & I_{n-1} \end{array} \right]$
is defined and invertible on $K_S$. It follows that 
$\tilde{H}$ is defined and strictly positive definite on $K_S$.
Clearly, $f_{11} \tilde{H}$ is a matrix polynomial that is
stricly positive definite on $K_S$. By the induction hypothesis there
exist $s \in T$ and $U \in T_S^{n-1}$ such that 
\begin{equation}
(1+s)f_{11} \tilde{H} = I_{n-1}+U.
\end{equation}
On the other hand, there exists by $n=1$ elements $s_1,u_1 \in T$ such that 
\begin{equation}
(1+s_1)f_{11} =1+u_1.
\end{equation}
Rearrange equation (1) and multiply it by $f_{11}^3$ to get (with $I=I_{n-1}$)
\begin{equation}
f_{11}^3 \left[ \begin{array}{cc} f_{11} & g \\ g^T & H  \end{array} \right] =
\left[ \begin{array}{cc} f_{11} & g \\ 0 & f_{11}I \end{array} \right]^T 
\left[ \begin{array}{cc} f_{11}^2 & 0 \\ 0 & f_{11}\tilde{H} \end{array} \right]
\left[ \begin{array}{cc} f_{11} & g \\ 0 & f_{11}I \end{array} \right]
\end{equation}
Multiplying equation (4) by $(1+s)(1+s_1)^4$ and using equations (2) and (3), we get:
\begin{eqnarray*}
(1+s)(1+s_1)(1+u_1)^3 \left[ \begin{array}{cc} f_{11} & g \\ g^T & H  \end{array} \right] = 
\left[ \begin{array}{cc} 1+u_1 & (1+s_1)g \\ 0 & (1+u_1)I \end{array} \right]^T \cdot \\ \cdot
\left[ \begin{array}{cc} (1+s)(1+u_1)^2 & 0 \\ 0 & (1+s_1)^2 (I+U) \end{array} \right] 
\left[ \begin{array}{cc} 1+u_1 & (1+s_1)g \\ 0 & (1+u_1)I \end{array} \right]
\end{eqnarray*}
Since 
\begin{equation*}
\left[ \begin{array}{cc} (1+s)(1+u_1)^2 & 0 \\ 0 & (1+s_1)^2 (I+U) \end{array} \right] =I_n+W
\end{equation*}
for some $W \in T^n_S$, it follows that
\begin{eqnarray*}
& (1+s)(1+s_1)(1+u_1)^3 \left[ \begin{array}{cc} f_{11} & g \\ g^T & H  \end{array} \right] = & \\ 
& =\left[ \begin{array}{cc} 1+u_1 & (1+s_1)g \\ 0 & (1+u_1)I \end{array} \right]^T 
\left[ \begin{array}{cc} 1+u_1 & (1+s_1)g \\ 0 & (1+u_1)I \end{array} \right]+W' &
\end{eqnarray*}
for some $W' \in T^n_S$. Write $\tilde{g}= (1+s_1)g$. 
By Lemma \ref{lmst} 
%(applied to $M=\sum M_n(\RR[\underline{X}])^2$ and $p=1+\sum_{i=1}^d X_i^2$) 
there exists an element $c \in \sum \RR[\underline{X}]^2$ 
%(of the form $kp^l$ for $k,l \in \NN$) 
such that 
\begin{equation}
c I_{n-1}-\tilde{g}^T \tilde{g} =: \sigma \in \sum M_{n-1}(\RR[\underline{X}])^2.
\end{equation}
Write $v=1+c$, $W''=v(1+v)W'$ and note that
\begin{eqnarray*}
& v(1+v)(1+s)(1+s_1)(1+u_1)^3 \left[ \begin{array}{cc} f_{11} & g \\ g^T & H  \end{array} \right] = &  
 \\ 
& = v(1+v)\left[ \begin{array}{cc} (1+u_1)^2 & (1+u_1)\tilde{g} \\
(1+u_1)\tilde{g}^T & \tilde{g}^T \tilde{g}+(1+u_1)^2 I \end{array} \right] +W'' = & 
\\ 
& =  \left[ \begin{array}{cc} v(1+u_1)^2 & 0 \\ 
0 & (v(1+u_1)^2 +v^2(2u_1+u_1^2)+1)I+(v+1)\sigma \end{array} \right] + & 
\\ 
& +
\left[ \begin{array}{cc} v (1+u_1) & (1+v)\tilde{g} \\ 0 & 0 \end{array} \right]^T
\left[ \begin{array}{cc} v (1+u_1) & (1+v)\tilde{g} \\ 0 & 0 \end{array} \right]
+W''  &
\end{eqnarray*}
which clearly belongs to $I_n+T_S^n$. It is also clear that $v(1+v)(1+s)(1+s_1)(1+u_1)^3$
belongs to $1+T_S$. 
\end{proof}

\section{Open problems}

\begin{enumerate}
\item Extend Theorem \ref{mst} to the case of matrix constraints. 

This problem is suggested in \cite{ks}. They extend Hol-Scherer theorem this way.

\item Extend Krivine's nichtnegativstellensatz ($f \ge 0$ on $K_S$ iff $ft=f^{2k}+u$ for some $t,u \in T_S$ and $k \in \NN$)
to matrix polynomials. 

A possible approach is given in Section 4.2 of \cite{sch2}. 
The matrix version of the Hilbert's 17th problem (i.e. the case $S=\emptyset$) was proved independently in \cite{gr} and \cite{ps}. 
For a constructive proof see Proposition 10 in \cite{sch2}.

\item Suppose that for some $S$ every $f \in \RR[\underline{X}]$ such that $f \ge 0$ on $K_S$ belongs to $T_S$. Does it follow that
every symmetric $F$ in $M_n(\RR[\underline{X}])$ which is positive semidefinite on $K_S$ belongs to $T_S^n$?

This is true in the following one-dimensional cases: $S=\emptyset$ by \cite{j}, \cite{d} or \cite{glr} and $S=\{X\}$, $S=\{X,1-X\}$ 
 by \cite{ds}.
\end{enumerate}

\end{document}